\documentclass[11pt,a4paper]{amsart}
\usepackage[utf8]{inputenc}
\usepackage[T1]{fontenc}
\usepackage{textcomp}
\usepackage{lmodern}
\usepackage{amssymb}
\usepackage{amsmath}
\usepackage{bbm}
\usepackage[normalem]{ulem}
\usepackage{hyperref}
\usepackage{comment}
\excludecomment{hide}
\includecomment{keep}

\usepackage{hyperref}
\usepackage{xcolor}
\hypersetup{%
    colorlinks=false,
    pdfborderstyle={/S/U/W 1}
}
\usepackage{tikz}

\definecolor{egraf}{rgb}{0.4,0.6,0.4}

\usepackage{xspace}
\usepackage{enumitem}
\usepackage{bbm}
\setenumerate[1]{label=(\roman{*}), ref=(\roman{*})}
\setenumerate[2]{label=(\alph{*}), ref=(\alph{*})}

\theoremstyle{plain}
\newtheorem{thm}{Theorem}[section]
\newtheorem*{thm*}{Theorem}
\newtheorem*{lem*}{Lemma}
\newtheorem{cor}[thm]{Corollary}
\newtheorem{prop}[thm]{Proposition}
\newtheorem{lem}[thm]{Lemma}

\newtheorem{quest}[thm]{Question}
\newtheorem{claim}[thm]{Claim}
\theoremstyle{definition}
\newtheorem{defn}[thm]{Definition}
\newtheorem{exmp}[thm]{Example}

\theoremstyle{remark}
\newtheorem{rem}[thm]{Remark}

\newcommand*{\claimproofname}{Proof of claim}
\newenvironment{claimproof}[1][\claimproofname]{\begin{proof}[#1]}{\end{proof}}

\newcommand{\eps}{\varepsilon}
\newcommand{\vertiii}[1]{{\left\vert\kern-0.25ex\left\vert\kern-0.25ex\left\vert #1 
    \right\vert\kern-0.25ex\right\vert\kern-0.25ex\right\vert}}

\DeclareMathOperator{\linspan}{span}
\DeclareMathOperator{\clinspan}{\overline{span}}

\DeclareMathOperator{\sgn}{sign}

\DeclareMathOperator{\supp}{supp}

\newcommand\restr[2]{{
    \left.\kern-\nulldelimiterspace 
      #1 
      \littletaller 
    \right|_{#2} 
  }}
\newcommand{\littletaller}{\mathchoice{\vphantom{\big|}}{}{}{}}

\newcommand{\HB}{\text{H{\kern -0.35em}B}}

\begin{document}

\title{Strictly convex norms and the local diameter two property}

\author[T.~A.~Abrahamsen]{Trond A.~Abrahamsen}
\address[T.~A.~Abrahamsen]{Department of Mathematics, University of
  Agder, Postboks 422, 4604 Kristiansand, Norway.}
\email{trond.a.abrahamsen@uia.no}
\urladdr{http://home.uia.no/trondaa/index.php3}

\author[P.~H\'ajek]{Petr H\'ajek}
\address[P.~H\'ajek]{Department of Mathematics\\Faculty of Electrical
Engineering\\Czech Technical University in Prague\\Technick\'a 2, 166 27
Praha 6\\ Czech Republic\\ \href{https://orcid.org/0000-0002-1714-3142}{ORCiD: 0000-0002-1714-3142}}
\email{hajekpe8@fel.cvut.cz}

\author[V.~Lima]{Vegard Lima}
\address[V.~Lima]{Department of Mathematics, University of
  Agder, Postboks 422, 4604 Kristiansand, Norway.}
\email{Vegard.Lima@uia.no}

\author[S.~Troyanski]{Stanimir Troyanski}
\address[S.~Troyanski]{Institute of Mathematics and Informatics,
  Bulgarian Academy of Science, bl.8, acad. G.~Bonchev str.~1113
  Sofia, Bulgaria}
\email{stroya@um.es}

\begin{abstract}
  We introduce and study a strict monotonicity property of the norm in
  solid Banach lattices of real functions that prevents such spaces
  from having the local diameter two property.
  Then we show that any strictly convex 1-symmetric norm on
  $c_0(\Gamma)$ possesses this property.
  In the opposite direction, we show that any Banach space which is
  strictly convex renormable and contains a complemented copy of
  $c_0(\mathbb N),$ admits an equivalent strictly convex norm for
  which the space has the local diameter two property.  In particular,
  this enables us to construct a strictly convex norm on
  $c_0(\Gamma),$ where $\Gamma$ is uncountable, for which the space
  has a 1-unconditional basis and the local diameter two property.
\end{abstract}

\maketitle

\section{Introduction}
\label{sec:intro}
A Banach space (or a norm $\|\cdot\|$ on a Banach space) $X$ is said
to be \emph{strictly convex} if for all $x, y \in X$ with
$\frac{1}{2}\|x + y\| = \|x\| = \|y\|,$ we have $x = y.$ Another way of
stating this is that the unit sphere $S_X$ of $X$ does not contain any
non-trivial line segments.

A lot of effort has gone into understanding when a Banach space admits
an equivalent strictly convex norm.  It is known for example that
every separable Banach space admits an equivalent strictly convex norm
\cite[Theorem~9]{MR1501880}. For non-separable spaces things get much
more murky.  Day showed that when $\Gamma$ is an uncountable set,
$\ell_\infty(\Gamma)$ cannot be renormed to have a strictly convex
norm, while $c_0(\Gamma)$ can \cite{MR67351}. Uncountability of
$\Gamma$ is important here as $\ell_\infty(\mathbb N)$ is strictly
convex renormable (see \cite[Remark~151 and
Proposition~152]{MR4501243}).  Haydon \cite{MR1041141} showed that
there exist uncountable scattered compacts $K$ for which $C(K)$ fails
to be strictly convex renormable. At the present moment there does not
exist an effective (checkable) criterion for when a (non-separable)
Banach space admits an equivalent strictly convex norm or when it does
not.  For more on this topic the introduction in \cite{MR2876968} is a
good place to start.

Recall that a \emph{slice} of the unit ball $B_X$ of a Banach space $X$
is a set of the form
\begin{equation*}
  S(x^*, \eps) := \{x \in B_X: x^*(x) > 1 - \eps\},
\end{equation*}
where $x^* \in S_{X^*}$ and $\eps > 0.$
A Banach space $X$ is said to have the
\emph{local diameter two property (LD2P)}
provided every slice of $B_X$ has diameter two.
While the usual focus of renorming theory is to find an equivalent
norm with the nicest properties possible our focus is different.
We will study to what degree the LD2P and strict convexity can coexist.
While at first sight the LD2P seems to be incompatible with strict
convexity, there are examples showing that these properties indeed can
coexist.
In fact, even stronger properties are compatible with strict convexity.
A Banach space is \emph{almost square}
if for every finite subset $x_1,\ldots,x_n \in S_X$
and every $\varepsilon > 0$ there exists $y \in S_X$
such that $\|x_i + y\| \le 1 + \varepsilon$ for all $i = 1, \ldots, n$.
Any almost square Banach space has the LD2P
(see e.g \cite[Theorem~1.3 and Proposition~2.5]{MR3415738}).
Recall that a Banach space $X$ is \emph{M-embedded} if
$X^{***} = X^{*} \oplus_1 X^\perp$, that is
$X$ is an M-ideal in its bidual.
Note that non-reflexive M-embedded spaces are almost square
\cite[Corollary~4.3]{MR3415738}.
The quotient $C(\mathbb{T})/A$, where $C(\mathbb{T})$ is the space of
continuous functions on the complex unit circle $\mathbb{T}$ and where
$A$ is the disc algebra, is an example of a strictly convex
Banach space with the LD2P.
In fact, $C(\mathbb{T})/A$ is M-embedded and even has a smooth dual
\cite[Remark~IV.1.17]{MR1238713}.
Another example can be constructed as follows:
Let $\varphi$ be the function on $c_0(\mathbb N)$ defined by
$\varphi(x) = \sum_{n=1}^\infty x^{2n}_n$ where $x = (x_n).$
The \emph{Nakano norm} on $c_0(\mathbb N)$ is defined by
\begin{equation*}
  \|x\| = \inf\{\lambda > 0 : \varphi(x/\lambda) \le 1\}
\end{equation*}
for every $x \in c_0(\mathbb N)$.  The space $c_0(\mathbb{N})$ with
the Nakano norm is strictly convex and is almost square
\cite{MR3499106}.  A third and recent example is the strictly convex
renorming of $L_1[0,1]$ constructed in \cite{MR4696078}.  This example
is not almost square, but has the LD2P.

While the LD2P and strict convexity can coexist in a Banach space $X$,
strict convexity in $X^{**}$ is asking too much, at least if we ask
$X$ to have the property that all finite convex combinations of slices
of its unit ball have diameter two. (Whether the same conclusion holds
only assuming the LD2P, is an open question). Indeed, in this case not
only does $X^{**}$ fail to be strictly convex, it also fails to be
smooth \cite[Corollary~2.6]{MR3574595}. Another result in this
direction is that if $X$ has a bimonotone basis and the property that
all non-empty relatively weakly open subsets of its unit ball have
diameter two, then $X^{**}$ is not strictly convex
\cite[Proposition~2.10]{MR3574595}.
These results point to the fact that
constructing strictly convex Banach spaces with the LD2P is not
completely straightforward.

Note that all the examples above with the LD2P and strictly convex
norms, are separable.  If $\Gamma$ is an uncountable set we know that
Day's norm on $c_0(\Gamma)$ is strictly convex (and even locally
uniformly rotund), but it is not immediately obvious how to give
$c_0(\Gamma)$ a strictly convex norm with the LD2P.  If the norm of a
Banach space $X$ is locally uniformly rotund then every $x \in S_X$ is
strongly exposed, so the LD2P and local uniformly rotundity cannot
coexist. In the present paper, we construct the first,
to the best of our knowledge,
example of a strictly convex renorming of
$c_0(\Gamma)$ with the LD2P.

The rest of the paper is organized as follows: After some preliminary
results in Section~\ref{sec:prelim}, we introduce in
Section~\ref{sec:a-uniform-strict-monotonicity-property}
a strict monotonicity property of the norm in solid Banach lattices of
real functions and show that this property cannot coexist with the
LD2P.

In Section~\ref{sec:1uc-uniformly-strictly-monotone} we show that
under some assumptions on the norm, solid Banach lattices of real
functions possess the strict monotonicity property just mentioned.  In
particular, we show that for any set $\Gamma,$ any strictly convex
norm on $c_0(\Gamma)$ with a 1-symmetric basis, fails the LD2P.

Finally, in Section~\ref{sec:rotund-norm-c_0gamma} we prove that every
strictly convex renormable Banach space which contains a complemented copy
of $c_0(\mathbb N)$ has an equivalent norm which is both strictly convex and
almost square. In particular, this enables us to prove that for any infinite
set $\Gamma$, $c_0(\Gamma)$ has a strictly convex renorming with a
1-unconditional basis that is almost square.

\section{Preliminaries}
\label{sec:prelim}
Let us start by recalling some definitions and results that we will
need later in the paper.

A \emph{solid Banach lattice of real functions} on some set $\Gamma$ is a
Banach space $(X,\|\cdot\|)$ consisting of real-valued bounded functions
$x: \Gamma \to \mathbb R$ with the following properties
\begin{itemize}
\item
  if $x \in X$ and $y: \Gamma \to \mathbb R$ with $|y(\gamma)| \le |x(\gamma)|$
  for all $\gamma \in \Gamma$, then $y \in X$ and
  $\|y\| \le \|x\|$;
\item for every $\alpha \in \Gamma,$ the function $e_\alpha := \delta_\alpha$,
  where $\delta_\alpha(\gamma) = 1$ if $\gamma = \alpha,$ and
  $\delta_\alpha(\gamma) = 0$ if $\gamma \neq \alpha$, belongs to $X$.
\end{itemize}
We will assume that $\|e_\gamma\| = 1$ for every $\gamma \in \Gamma.$
Hence $\ell_1(\Gamma) \subseteq X \subseteq \ell_\infty(\Gamma)$
and $\|\cdot\|_\infty \le \|\cdot\| \le \|\cdot\|_1$.
When $\Gamma$ is the set $\mathbb N$ of natural numbers, we obtain the
well known class of K\"othe sequence spaces.

For $\gamma \in
\Gamma,$ we let $e_\gamma^* \in
X^*$ denote the biorthogonal functional to
$e_\gamma.$ Note that for $x \in X$ we have $e_\gamma^*(x) =
x(\gamma).$

Clearly $(e_\gamma)_{\gamma \in \Gamma}$ is a 1-unconditional
\emph{basic set} in $X,$ i.e. a 1-unconditional basis for the norm
closure of the linear span of $\{e_\gamma: \gamma \in \Gamma\}$ in
$X.$ Hence if $c_{00}(\Gamma)$, the space of finitely supported
real-valued functions on $\Gamma,$ is dense in $X$, the unit vectors
$(e_\gamma)_{\gamma \in \Gamma}$ form a 1-unconditional basis for
$X$. Recall that if $\Gamma$ is a non-empty set, then
$(e_\gamma)_{\gamma \in \Gamma}$ is called \emph{an unconditional
  basis} for a Banach space $X$ if for every $x \in X$ there is a
unique family of real numbers $(a_\gamma)_{\gamma \in \Gamma}$ such
that $x = \sum_{\gamma \in \Gamma} a_\gamma e_\gamma$ in the sense
that for every $\eps > 0$ there is a finite set $F \subset \Gamma$
such that $\|x - \sum_{\gamma \in G} a_\gamma e_\gamma\| < \eps$ for
every $G \supset F.$ If moreover, for any finite set $F$ in $\Gamma$
and any set of real numbers $(a_\gamma)_{\gamma \in F}$ and
$(b_\gamma)_{\gamma \in F}$ with $|b_\gamma| \le |\alpha_\gamma|$ for
$\gamma \in F,$ we have
\[
  \left\|\sum_{\gamma \in F} b_\gamma e_\gamma\right\| \le
  \left\|\sum_{\gamma \in F} a_\gamma e_\gamma\right\|,
\]
then we say that $(e_\gamma)_{\gamma \in \Gamma}$ is
\emph{1-unconditional}.

If $X$ is a solid Banach lattice of real functions on $\Gamma,$ we say
that $X$ is \emph{1-symmetric} if for every $x = x(\gamma) \in X$ and
every permutation $\pi: \Gamma \to \Gamma$ we have
$x_\pi = x(\pi(\gamma)) \in X$ and $\|x_\pi\| = \|x\|$. We say that
the norm of the lattice $X$ is \emph{strictly monotone} if
$\|y\| < \|x\|$ whenever $|y| < |x|$, that is if
$|y(\gamma)| \le |x(\gamma)|$ for all $\gamma \in \Gamma$ and
$|y(\beta)| < |x(\beta)|$ for some $\beta \in \Gamma$.

Let $(e_\gamma)_{\gamma \in \Gamma}$ be a 1-unconditional basis for a
Banach space $X.$ For each $\gamma \in \Gamma$ let $e_\gamma^* \in X^*$
denote the biorthogonal functional of $e_\gamma.$ We say that
$(e_\gamma, e_\gamma^*)_{\gamma \in \Gamma}$ is \emph{shrinking} if
$\clinspan\{e_\gamma^*: \gamma \in \Gamma\} = X^*.$

The following results will be used several times in the paper. The
first one is James' classical representation of the bidual of a Banach
space with a shrinking basis in case the shrinking basis is
uncountable and unconditional. We include a proof for easy reference.

\begin{prop}
  \label{prop:shrinking-1unc-sup}
  Let $X$ be a Banach space with a shrinking 1-unconditional basis
  $(e_\gamma, e_\gamma^*)_{\gamma \in \Gamma}$, and $\mathcal F$ the set
  $\{F \subset \Gamma: |F| < \infty\}$ ordered by inclusion. Then for every
  $x^{**} \in X^{**}$ we have
  \[
    \|x^{**}\| = \sup_{F \in \mathcal F} \left\|\sum_{\gamma \in
        F}x^{**}(e_\gamma^*)e_\gamma\right\| = \lim_{\mathcal F}
    \left\| \sum_{F \in \mathcal F}x^{**}(e_\gamma^*)e_\gamma\right\|.
  \]
\end{prop}

\begin{proof}
  Let $F \subset \Gamma$ be finite and $P_F:X \to X$ the projection
  onto $\linspan\{e_\gamma: \gamma \in F\}$ defined by
  \[
    P_F(x)
    = \sum_{\gamma \in F} e_\gamma^*(x) e_\gamma.
  \]
  Since the basis is shrinking, we have 
  \begin{align}
    \label{eq:15}
     P_F^{**}(x^{**})
    = \sum_{\gamma \in F} x^{**}(e_\gamma^*) e_\gamma,
  \end{align}
  for every $x^{**} \in X^{**}.$
  Indeed if $\gamma_0 \in \Gamma$ and $x^* = e_{\gamma_0}^*,$ then
  \begin{align}
    \label{eq:16}
    \langle P_F^{**}(x^{**}), x^*\rangle
    = \langle x^{**}, P_F^*(e_{\gamma_0}^*)\rangle
    =
      \begin{cases}
        x^{**}(e_{\gamma_0}^*),& \gamma_0 \in F\\
        0, & \gamma_0 \notin F,
      \end{cases}
  \end{align}
  so
  \[
    \langle P_F^{**}(x^{**}), x^*\rangle
    = \left\langle \sum_{\gamma \in F}
    x^{**}(e_\gamma^*) e_\gamma, x^*\right\rangle.
  \]
  From this in tandem with linearity and continuity of the functionals
  involved and the fact that
  $\clinspan\{e_\gamma^*: \gamma \in \Gamma\} = X^*,$ the equality
  (\ref{eq:15}) follows.  Moreover, since the basis is
  1-unconditional,
  \[
    \|x^{**}\|
    = \sup_{F \in \mathcal F} \|P_F^{**}x^{**}\|.
  \]
  Indeed, because of
  (\ref{eq:16}) and since
  $\clinspan \{e_\gamma^*: \gamma \in \Gamma\} = X^{*},$ the net
  $(P_F^{**} x^{**})_{F \in \mathcal F},$
  converges weak$^*$ to $x^{**}.$ Thus by weak$^*$ lower
  semi-continuity of the norm we have 
  $\|x^{**}\| \le \liminf_{\mathcal F} \|P_F^{**}x^{**}\| \le
  \sup_{F \in \mathcal F}\|P_F^{**}x^{**}\|.$ Also, since
  $(e_\gamma, e_\gamma^*)_{\gamma \in \Gamma}$ is 1-unconditional (in
  particular monotone) we have $\|P_F\| = 1,$ so we can conclude
  \[
    \|x^{**}\|
    = \sup_{F \in \mathcal F} \|P_F^{**}x^{**}\|
    = \lim_{\mathcal F}\|P_F^{**}x^{**}\|.
    \qedhere
  \]
\end{proof}

The next result is well known.
For a proof see that of Theorem~2.6 (ii) implies (iii)
in \cite{MR2352727}.

\begin{prop}
  \label{prop:sc=>sm}
  Let $X$ be a solid Banach lattice of real functions.
  If the norm of $X$ is strictly convex,
  then $X$ is strictly monotone.
\end{prop}

The converse does not hold in general, just consider $X = \ell_1$.
But if $(X,\|\cdot\|)$ is a solid Banach lattice of real functions on some
set $\Gamma$ such that $\|\cdot\|_\infty \le \|\cdot\|,$ we have
from \cite[Theorem~2.6]{MR2352727}, that the following statements are
equivalent:
\begin{enumerate}
\item $X$ admits a pointwise lower semi-continuous strictly convex norm;
\item $X$ admits a lattice pointwise lower semi-continuous strictly
  convex norm;
\item $X$ admits a pointwise lower semi-continuous strictly monotone
  norm (in \cite{MR2352727} the term strictly lattice norm is used).
\end{enumerate}

\section{A uniform strict monotonicity property of the
  norm}
\label{sec:a-uniform-strict-monotonicity-property}

Throughout this section $X$ will be a solid Banach lattice of real
functions on a set $\Gamma$ with a (normalized) 1-unconditional basic
set $(e_\gamma)_{\gamma \in \Gamma}.$
In this section we will introduce a property that will
prevent $X$ from having the LD2P.
The Nakano norm on $c_0$ shows that a strictly monotone norm is
not enough to prevent the LD2P.
On the other hand, the space $c_0 \oplus_1 \mathbb{R}$
has a 1-unconditional basis, fails the LD2P, and the norm
is not strictly monotone.
The last example can serve as a motivation for our property.

\begin{defn}
  For each $\alpha \in \Gamma$ and $\varepsilon > 0$ define
  \begin{equation*}
    E_\alpha(\varepsilon)
    := \sup\left\{\|x - e_\alpha^*(x) e_\alpha \|:
      x \in S_X, e_\alpha^*(x) > 1 - \varepsilon \right\}.
  \end{equation*}
  If there exist $\alpha \in \Gamma$ and
  $\varepsilon > 0$ such that $E_\alpha(\varepsilon) < 1$, then we say
  that $X$ is \emph{uniformly strictly monotone at
    $\alpha \in \Gamma$}.
\end{defn}
We note that $E_\alpha$ is a non-decreasing function of $\eps$ since
we are taking supremum over a smaller set.

\begin{prop}
  \label{prop:unif-sm-not-ld2p-v2}
  If there exists $\alpha \in \Gamma$ such that
  $X$ is uniformly strictly monotone at $\alpha$,
  then $X$ fails to have the LD2P.
\end{prop}

\begin{proof}
  Assume $X$ is uniformly strictly monotone at $\alpha \in \Gamma.$
  Let $\varepsilon > 0$ such that $E_\alpha(\varepsilon) < 1$.  We may
  and will assume that $\varepsilon < 2/3$.  We will show that the
  diameter of the slice
  $S(e_\alpha^*,\varepsilon)$ is
  bounded away from $2$.
  \begin{claim}
    There exists a $C > 0$ such that
    \begin{equation*}
      \left\|
        x - \left(e_\alpha^*(x) - \frac{\varepsilon}{2} \right)e_\alpha
      \right\|
      \le 1 - C
    \end{equation*}
    for all $x \in S(e_\alpha^*, \varepsilon)$.
  \end{claim}
  \begin{claimproof}
    Let $C := (1 - 3\varepsilon/2)(1 - E_\alpha(\varepsilon)) > 0$.

    If $x \in S(e_\alpha^*, \varepsilon)$
    we set
    \begin{equation*}
      \lambda =
      \frac{e_\alpha^*(x) - \varepsilon/2}{e_\alpha^*(x)}.
    \end{equation*}
    Then, since $e_\alpha^*(x) \le 1$ and $\varepsilon < 2/3$,
    \begin{align*}
      1
      > \lambda
      > \frac{1 - \varepsilon - \varepsilon/2}{e_\alpha^*(x)}
      > 1 - \frac{3}{2}\varepsilon
      > 0
    \end{align*}
    and
    \begin{equation*}
      x - \left(e_\alpha^*(x) - \frac{\varepsilon}{2} \right)e_\alpha
      =
      (1 - \lambda) x
      +
      \lambda (x - e_\alpha^*(x)e_\alpha),
    \end{equation*}
    so by convexity of the norm
    \begin{align*}
      \left\|
        x - \left(e_\alpha^*(x) - \frac{\varepsilon}{2} \right)e_\alpha
      \right\|
      &\le
      (1-\lambda)\|x\|
      +
      \lambda
      \|x - e_\alpha^*(x)e_\alpha\| \\
      &\le
        1 - \lambda + \lambda E_\alpha(\varepsilon)
        \le
        1 - C
    \end{align*}
    as claimed.
  \end{claimproof}
  Let $y, z \in  S(e_\alpha^*, \varepsilon)$.
  We use unconditionality,
  i.e. $\|u\| \le \|v\|$ whenever $|u| \le |v|$,
  and the above claim and get
  \begin{align*}
    \|y - z\|
    &=
      \left\|
      e_\alpha^*(y - z) e_\alpha
      + \bigl(y - e_\alpha^*(y)e_\alpha \bigr)
      - \bigl(z - e_\alpha^*(z)e_\alpha \bigr)
      \right\|\\
    &\le
      \left\|\varepsilon e_\alpha
      + \bigl(y - e_\alpha^*(y)e_\alpha\bigr)
      - \bigl(z - e_\alpha^*(z)e_\alpha\bigr)
      \right\|\\
    &\le
      \left\|
      \bigl(y - e_\alpha^*(y)e_\alpha \bigr)
      +
      \frac{\varepsilon}{2} e_\alpha
      \right\| +
      \left\|
      \bigl(z - e_\alpha^*(z)e_\alpha \bigr)
      -
      \frac{\varepsilon}{2} e_\alpha
      \right\|\\
    &=
      \left\|
      y -
      \left(e_\alpha^*(y) - \frac{\varepsilon}{2}\right) e_\alpha
      \right\| +
      \left\|
      z -
      \left(e_\alpha^*(z) - \frac{\varepsilon}{2}\right) e_\alpha
      \right\|\\
    &\le
      2(1 - C).
  \end{align*}
  Hence the diameter of $S(e_\alpha^*,\varepsilon)$ is strictly
  less than $2$.
\end{proof}

\section{Spaces with 1-unconditional bases that have
  uniformly strictly monotone norms}
\label{sec:1uc-uniformly-strictly-monotone}
Let $X$ be a solid Banach lattice of real functions on a set
$\Gamma$ with a (normalized) 1-unconditional basic set
$(e_\gamma)_{\gamma \in \Gamma}.$ For $A \subset \Gamma,$ let
$\mathbbm{1}_A \in \ell_\infty(\Gamma)$ denote the characteristic
function on $A.$ Note that $\mathbbm{1}_A \in X$ if $A$ is finite.  We
define
\begin{equation*}
  \begin{aligned}
    \mathfrak{L_F}
  =
  \sup\left\{
    \left\|\mathbbm{1}_A\right\| 
    : A \subset \Gamma, |A| < \infty
  \right\} \text{ and }
  \mathfrak{L}
  =
    \left\|\mathbbm{1}_\Gamma\right\|. 
  \end{aligned}
\end{equation*}
Clearly $\mathfrak{L} \ge \mathfrak{L_F}.$ Also, since
$(e_\gamma)_{\gamma \in \Gamma}$ is normalized we have
$\mathfrak{L_F} \ge 1$, and if $X$ is strictly monotone we have
$\mathfrak{L_F} > 1$. When we write $\mathfrak{L} < \infty,$ we
implicitly assume that $\mathbbm{1}_\Gamma$ is an element in $X.$
Clearly $\|\mathbbm{1}_\Gamma\| < \infty$ is equivalent to $X$ being lattice isomorphic to $\ell_\infty(\Gamma).$

For a finite subset $F$ of $\Gamma$ let $P_F: X \to X$ be the
natural projection onto the set $\linspan\{e_\gamma: \gamma \in F\}.$

If $F \subset \Gamma$ is finite and $(a_\gamma)_{\gamma \in F}$ is a
sequence of scalars, then
\begin{equation*}
  \max_{\gamma \in F} |a_\gamma|
  \le \left\| \sum_{\gamma \in F} a_\gamma e_\gamma \right\|
  \le \max_{\gamma \in F} |a_\gamma| \cdot \left\| \sum_{\gamma \in F}
    e_\gamma \right\|
  \le \mathfrak{L_F} \cdot \max_{\gamma \in F}|a_\gamma|.
\end{equation*}
Thus $\clinspan\{e_\gamma: \gamma \in \Gamma\}$ is isomorphic to $c_0(\Gamma)$
if and only if $\mathfrak{L_F} < \infty$.

Let $X$ be $c_0(\mathbb{N})$ with the Nakano norm. The canonical basis
for $c_0(\mathbb{N})$, $(e_n)_{n=1}^\infty$, is then a normalized strictly
monotone 1-unconditional basis for $X$. If we want to show that $X$
has the LD2P, then we can make use of the following fact: For every
$j \in \mathbb N$ we have that for any fixed $r \in (0,1)$
\begin{equation*}
  \inf_{n \neq j}
  \left\{
    \|e_j + r e_n\|
  \right\} = 1.
\end{equation*}
In particular, this holds for $r = 1/\mathfrak{L_F}$.  Actually, we
can replace $e_j$ with any $x \in S_X$ with finite support in the
infimum above. Hence inside every slice $S(x^*, \eps)$ of $B_X$ we can
for every $r \in (0, 1)$ find elements of the form $x \pm re_n.$
This tells us that every such slice has diameter two. We will return
to this example in Section~\ref{sec:rotund-norm-c_0gamma}.

\begin{thm}
  \label{thm:main-1uncond}
  Let $X$ be a strictly monotone solid Banach lattice of real
  functions on the set $\Gamma$ with a basic set
  $(e_\gamma)_{\gamma \in \Gamma}.$
  \begin{enumerate}
  \item\label{item:thm-main-1uncond-pt1}
    Assume $\mathfrak{L_F} < \infty$ and for every $x \in X$ we
    have $\|x\| = \sup_{F \subset \Gamma, |F| < \infty} \|P_Fx\|.$ If there
    exists $\alpha \in \Gamma$ such that
  \begin{equation*}
    \inf_{\beta \neq \alpha} \left\{
      \left\|e_\alpha + \frac{1}{\mathfrak{L_F}} e_\beta\right\|
    \right\} > 1,
  \end{equation*}
  then $X$ fails the LD2P.  In particular, any equivalent renorming of
  $c_0(\Gamma)$ such that the canonical basis $(e_\gamma)_{\gamma \in \Gamma}$
  is strictly monotone and normalized, and there exists
  $\alpha \in \Gamma$ with
  \begin{equation*}
    \inf_{\beta \neq \alpha} \left\{
      \left\|e_\alpha + \frac{1}{\mathfrak{L_F}} e_\beta\right\|
    \right\} > 1,
  \end{equation*}
    fails the LD2P.
  \item\label{item:thm-main-1uncond-pt2} Assume that $X$ is lattice
    isomorphic to $\ell_\infty(\mathbb N).$ If there exists
    $k \in \mathbb N$ such that
  \begin{equation*}
    \inf_{n \neq k} \left\{
      \left\|e_k + \frac{1}{\mathfrak{L}} e_n\right\|
    \right\} > 1,
   \end{equation*}
   then $X$ fails the LD2P.
  \end{enumerate}
\end{thm}

\begin{rem}
  Note that it follows from Day's proof in \cite{MR67351} that when
  $\Gamma$ is uncountable, $\ell_\infty(\Gamma)$ does not have an
  equivalent strictly monotone norm.
\end{rem}

\begin{proof}
  Let $\alpha \in \Gamma$.
  We aim to show that $X$ is uniformly strictly monotone at $\alpha$
  and the proof of \ref{item:thm-main-1uncond-pt1} and
  \ref{item:thm-main-1uncond-pt2} are the same up to a point.

  Assume $K > 1$ is a real number such that
  \begin{equation*}
    l := \inf_{\beta \neq \alpha} \left\{
      \left\|e_\alpha + \frac{1}{K} e_\beta\right\|
    \right\} > 1.
  \end{equation*}
  Fix $\eta < 1$ such that
  \begin{equation*}
    \eta >
    \frac{K + 1}{l K + 1}.
  \end{equation*}
  Note that since $l \le 2$
  \begin{equation*}
    \eta
    >
    \frac{K + 1}{l K + 1}
    >
    \frac{K + 1}{l (K + 1)}
    =
    \frac{1}{l}
    \ge
    \frac{1}{2}.
  \end{equation*}
  \begin{claim}\label{claim:no_ld2p_claim2}
    There exists $\mathfrak{a} < 1/K$ such that for all
    $\eta e_\alpha + a e_\beta \in B_X$ with $a \in \mathbb R$ and $\beta \ne
    \alpha,$ we have $|a| \le \mathfrak{a}$.
  \end{claim}
  \begin{claimproof}
    For notational reasons, let us introduce the function
    \begin{equation*}
      f(\tau)
      := \tau \eta l - (1 - \tau)\eta - (1 - \eta)\tau/K.
    \end{equation*}
    Rewriting the assumption
    \begin{equation*}
      \eta > \frac{K + 1}{lK + 1}
    \end{equation*}
    gives $f(1) = \eta l - (1 - \eta)/K > 1$.
    Since $f$ is continuous we can find $\theta \in (0, 1)$ such that
    \begin{equation*}
      f(\theta)
      = \theta \eta l - (1 - \theta)\eta - (1 - \eta)\theta/K
      > 1.
    \end{equation*}
    Set $\mathfrak{a} := \theta/K.$
    Let
    $\|\eta e_\alpha + a e_\beta\| \le 1$ for some
    $\beta \in \Gamma \setminus \{\alpha\}$ and $a \in \mathbb R.$
    Define
    \[
      g_\beta(t)
      := \left\|\eta e_\alpha + te_\beta\right\|.
    \]
    Since
    \[
      \eta e_\alpha + \theta e_\beta/K
      = \theta \eta (e_\alpha + e_\beta/K)
      + (1 - \theta)\eta e_\alpha + (1 - \eta) \theta e_\beta/K,
    \]
    we have
    \begin{align*}
      g_\beta(\theta/K)
      &= \|\eta e_\alpha + \theta e_\beta/K\|\\
      &\ge \theta \eta \|e_\alpha + e_\beta/K\|
        - (1 - \theta)\eta \|e_\alpha\| - (1 - \eta) \theta/K \|e_\beta\|\\
      &\ge \theta \eta l - (1 - \theta)\eta  - (1 - \eta) \theta/K\\
      &= f(\theta)\\
      &> 1
        \ge g_\beta(a).
    \end{align*}
    By 1-unconditionality of the basic set
    $(e_\gamma)_{\gamma \in \Gamma},$ we have that $g_\beta$ is even
    and is non-decreasing for $t \ge 0.$
    Hence $|a| \le \theta/K = \mathfrak{a}.$
  \end{claimproof}

  We are now ready to prove that $X$ 
  is uniformly strictly monotone at $\alpha$.
  Let $0 < \varepsilon < 1 - \eta$ be such that
  $\mathfrak{a} K < 1 - 2\varepsilon$.
  We will show that
  \begin{align*}
    E_\alpha(\eps)
    =
    \sup \{\|x - e_\alpha^*(x)e_\alpha\|: x \in S_X, e_\alpha^*(x) > 1 - \eps\}
    \le 1 - \varepsilon.
  \end{align*}
  Let $x \in S_X$ and assume $e_\alpha^*(x) > 1 - \eps$.
  Let $\beta \neq \alpha$.
  Now for some $z_\beta \in B_X$ we have
  \begin{equation*}
    x = e_\alpha^*(x) e_\alpha + e_\beta^*(x) e_\beta + z_\beta
  \end{equation*}
  and by strict monotonicity and $1$-unconditionality
  \begin{equation*}
    1 = \|x\|
    \ge
    \|e_\alpha^*(x) e_\alpha + e_\beta^*(x) e_\beta + z_\beta\|
    > \|\eta e_\alpha + e_\beta^*(x) e_\beta\|.
  \end{equation*}
  From Claim~\ref{claim:no_ld2p_claim2} we get that
  $|e_\beta^*(x)| \le \mathfrak{a} < 1/K$ for all $\beta \ne \alpha$.

  \ref{item:thm-main-1uncond-pt1}.
  We assume $\mathfrak{L_F} < \infty.$ Put $K:= \mathfrak{L_F},$
  and note that $\mathfrak{L_F} > 1$ since $X$ is strictly
  monotone. Therefore $K > 1,$ and we get by the assumption that there
  exists a finite subset $F \subset \Gamma$ such that
  \begin{align*}
    \|x - e_\alpha^*(x)e_\alpha\|
    &\le
      \left\|
      \sum_{\beta \in F \setminus \{\alpha\} } e_\beta^*(x) e_\beta
      \right\| + \varepsilon
      \le
      \mathfrak{a}
      \left\|\sum_{\beta \in F \setminus\{\alpha\}} e_\beta\right\|
      + \varepsilon\\
    &\le
      \mathfrak{a}
      \left\|\sum_{\beta \in F \cup \{\alpha\}} e_\beta\right\|
      + \varepsilon
      <
      \mathfrak{a}\mathfrak{L_F} + \varepsilon
      =
      1 - \varepsilon.
  \end{align*}
  This proves that $X$ is uniformly strictly monotone at $\alpha,$ and
  we have from Proposition~\ref{prop:unif-sm-not-ld2p-v2} that $X$
  fails the LD2P.

  For the particular case, let $X$ be an equivalent renorming of
  $c_0(\Gamma)$ such that the canonical basis
  $(e_\gamma)_{\gamma \in \Gamma}$ is strictly monotone and
  normalized. We then have that $X$ is a strictly monotone solid Banach
  lattice consisting of functions $x: \Gamma \to \mathbb R.$ Now we
  only need to note that under the assumptions on $X$ we have
  $\mathfrak{L_F} < \infty$ and from
  Proposition~\ref{prop:shrinking-1unc-sup} that
  $\|x\| = \sup_{F \subset \Gamma, |F| < \infty} \|P_Fx\|$ for every
  $x \in X.$

  \ref{item:thm-main-1uncond-pt2}.  We assume $\Gamma = \mathbb N$ and
  $\mathfrak{L} < \infty.$ Put $K:= \mathfrak{L}$ and note that
  $\mathfrak{L} > 1$ since $X$ is strictly monotone. Therefore
  $K > 1,$ and with $k = \alpha$ we get
  \begin{align*}
    \|x - e_k^*(x)e_k\|
    \le
      \left\|
      \mathfrak{a}\mathbbm{1}_{\supp(x) \setminus \{k\}}
      \right\|
      < \mathfrak{a}
      \left \|\mathbbm{1}_\Gamma
      \right\|
      = \mathfrak{a}\mathfrak{L}
      < 1 - \eps,
  \end{align*}
  so $E_k(\eps) < 1 - \eps$ and $X$ is therefore uniformly
  strictly monotone at $k.$
  As in \ref{item:thm-main-1uncond-pt1} we get that $X$ fails the
  LD2P by Proposition~\ref{prop:unif-sm-not-ld2p-v2}.
\end{proof}

\begin{thm}\label{thm:c_0-1symm-noLD2P}
  Let $\Gamma$ be a non-empty set and
  let $\|\cdot\|$ be an equivalent norm on
  $c_0(\Gamma)$ such that the canonical basis
  $(e_\gamma)_{\gamma \in \Gamma}$ is normalized
  and 1-symmetric.
  If $\|\cdot\|$ is strictly monotone, in particular if $\|\cdot\|$
  strictly convex, then it fails the LD2P.
\end{thm}

\begin{proof}
  Let $X$ be a renorming of
  $c_0(\Gamma)$ such that the canonical basis
  $(e_\gamma)_{\gamma \in \Gamma}$ is strictly monotone, normalized,
  and 1-symmetric. Choose any $\alpha \in \Gamma$.
  By Theorem~\ref{thm:main-1uncond}~\ref{item:thm-main-1uncond-pt1},
  it is enough to show that
  \begin{equation*}
    \inf_{\beta \neq \alpha} \left\{
      \left\|e_\alpha + \frac{1}{\mathfrak{L_F}} e_\beta\right\|
    \right\} > 1.
  \end{equation*}
  But by symmetry we have
  \begin{equation*}
   \left \|e_\alpha + \frac{1}{\mathfrak{L_F}} e_\beta\right\|
    =
    \left\|e_\alpha + \frac{1}{\mathfrak{L_F}} e_\gamma\right\|
  \end{equation*}
  for all $\beta,\gamma \in \Gamma \setminus \{\alpha\}$,
  so it is enough to observe that for $\beta \neq \alpha$
  \begin{equation*}
    \left\|e_\alpha + \frac{1}{\mathfrak{L_F}} e_\beta\right\|
    >
    \|e_\alpha\| = 1,
  \end{equation*}
  by strict monotonicity.

  The strictly convex case now follows from
  Proposition~\ref{prop:sc=>sm}.
\end{proof}

\begin{rem}
  It was shown in \cite[Corollary~6]{MR1348481} that
  $\ell_\infty(\mathbb N)$ admits no equivalent
  strictly convex 1-symmetric norm.
  It can also be shown that $\ell_\infty(\mathbb N)$ does not admit
  a strictly monotone 1-symmetric norm.
\end{rem}

\begin{quest}
  Let $X$ be a strictly monotone renorming of $c_0(\mathbb N)$
  such that the canonical basis $(e_n)_{n = 1}^\infty$ is normalized.
  Are then the following two conditions equivalent?
  \begin{enumerate}
    \item $X$ fails the LD2P;
    \item There exists $j \in \mathbb N$ such that
      $\inf_{n \neq j}
      \left\{
        \left\|e_j + \frac{1}{\mathfrak{L_F}} e_n\right\| \right\}
      > 1.$
  \end{enumerate}
\end{quest}

Let us end this section with some remarks about
the assumption $\mathfrak L_{\mathfrak{F}} < \infty$ in
Theorem~\ref{thm:main-1uncond}.

First let us give an example of a Banach space $X$
with a 1-symmetric basis and a strictly convex norm, which has the LD2P,
but which is not uniformly strictly monotone at any coordinate
and with $\mathfrak{L}_{\mathfrak{F}} = \infty$.

\begin{exmp}
  \label{ex:orlicz-SC1symLD2P}
  Define a function $M : [0,\infty) \to [0,\infty)$ by
  \begin{equation*}
    M(t)
    =
    \begin{cases}
      0, & t = 0;\\
      \frac{1}{4}e^2 \cdot e^{-1/t}, & t \in (0,\frac{1}{2}); \\
      t^2 , & t \ge \frac{1}{2}.
    \end{cases}
  \end{equation*}
  Then $M$ is a convex, non-decreasing continuous function
  satisfying $M(0) = 0$ and $\lim_{t \to \infty} M(t) = \infty$.
  If we consider the space $\ell_M$ of all sequences
  $x = (a_1,a_2,\ldots)$ of scalars such that
  $\Phi(x/\lambda) := \sum_{n=1}^\infty M(|a_n|/\lambda) < \infty$
  for some $\lambda > 0$, then $\ell_M$ is a Banach space with norm
  \begin{equation*}
    \|x\| = \inf \{
    \lambda > 0 : \sum_{n=1}^\infty M(|a_n|/\lambda) \le 1
    \}.
  \end{equation*}
  The unit vectors $\{e_n\}_{n=1}^\infty$ form
  a 1-symmetric basis of $h_M = \clinspan\{e_n\}_{n=1}^\infty$
  (see e.g. \cite[p.~115]{MR0500056}).
 
  Let us show that $X = h_M$ is also strictly
  convex, has the LD2P (it is actually M-embedded), is not uniformly
  strictly monotone at any coordinate, but
  $\mathfrak{L}_{\mathfrak{F}} = \infty$.

  A straightforward computation shows that
  $M$ is strictly convex, so $h_M$ is strictly convex
  (see e.g. \cite[p.~56]{MR1410390}).

  From Example~III.1.4 in \cite{MR1238713} we have that $h_M$
  is M-embedded since the complementary Orlicz function $M^*$
  satisfies the $\Delta_2$ condition at zero while $M$
  fails it.
  In particular, $h_M$ has the LD2P by \cite[Proposition~2.1]{MR2091678}.
  Let us prove the claimed properties of $M$ and $M^*$.
  We have that $M'(t) = M(t)\frac{1}{t^2}$ on $(0,\frac{1}{2})$
  and hence
  \begin{equation*}
    \lim_{t \to 0^+} \frac{t M'(t)}{M(t)}
    =
    \lim_{t \to 0^+} \frac{t M(t)/t^2}{M(t)}
    =
    \lim_{t \to 0^+} \frac{1}{t}
    = + \infty.
  \end{equation*}
  Since the above limit is infinite $M$ fails the $\Delta_2$-condition
  at zero (see e.g. \cite[p.~140]{MR0500056}).
  We have that $M'(t)$ is continuous and strictly increasing
  on $[0,\infty)$ and has a continuous and strictly increasing
  inverse $q(t)$.
  The complementary function to $M$ is given by
  $M^*(u) = \int_0^u q(t) \, dt$ and
  $M^*$ satisfies the $\Delta_2$-condition at zero
  (see the remarks following Proposition~4.b.2 in
  \cite[p.~148]{MR0500056})
  since the above limit is strictly greater than $1$
  (here we use that $q$ is continuous).

  Next let us show that $\mathfrak{L}_{\mathfrak{F}} = \infty$.
  For $s_n = \sum_{i=1}^n e_i$ in $h_M$ we have
  \begin{equation*}
    \Phi(s_n/\lambda)
    =
    \sum_{i=1}^n M(1/\lambda)
    = n M(1/\lambda).
  \end{equation*}
  It is well known that the infimum in the norm is attained
  so we need only solve $\Phi(s_n/\lambda) = 1$
  to get $M(1/\lambda) = 1/n$.
  For $n \ge 4$ we use
  $M(1/\lambda) = \frac{1}{4}e^2 e^{-\lambda}$ and get
  \begin{equation}\label{eq:norm_LF_hM_ex}
    \|s_n\| = \left\|\sum_{i=1}^n e_i \right\|
    = \lambda = 2 - 2\ln(2) + \ln(n)
  \end{equation}
  From this we see that $\mathfrak{L}_{\mathfrak{F}} = \infty$.

  Finally, we show that $h_M$ is not uniformly strictly monotone
  at any coordinate. By 1-symmetry of the basis we need
  only consider the first coordinate.

  Let $0 < d < 1/2$.
  For $n \in \mathbb{N}$ define
  \begin{equation*}
    x_n = (1-d)e_1 + \sum_{i=1}^n \frac{1}{k}e_{i+1}
  \end{equation*}
  where $k$ is chosen (depending on $n$) such that $\|x_n\| = 1$.
  Using \eqref{eq:norm_LF_hM_ex} we have
  \begin{equation*}
    \|x_n - e_1^*(x_n)e_1\|
    = \frac{1}{k}
    \left\|\sum_{i=1}^n e_{i+1}\right\|
    =
    \frac{\|s_n\|}{k}
    =
    \frac{2 - 2\ln(2) + \ln(n)}{k}.
  \end{equation*}
  Let us find $k$.
  Solve
  \begin{equation*}
    \Phi(x_n)
    =
    (1-d)^2 + \frac{n}{4}e^{2}e^{-k}
    =
    1
  \end{equation*}
  for $k$ to get
  \begin{equation*}
    e^{-k} = e^{-2}\frac{4}{n}(2d-d^2)
  \end{equation*}
  hence
  \begin{equation*}
    k = 2 - 2\ln(2) + \ln(n) - \ln(2d-d^2).
  \end{equation*}
  We therefore have
  \begin{align*}
    \|x_n - e_1^*(x_n)e_1\|
    &=
    \frac{2 - 2\ln(2) + \ln(n)}{2 - 2\ln(2) + \ln(n) - \ln(2d-d^2)}
    \\
    &=
    1 +
    \frac{\ln(2d-d^2)}{2 - 2\ln(2) + \ln(n) - \ln(2d-d^2)}.
  \end{align*}
  Since $d$ is fixed and $0 < d < 1/2$ we have
  that $\ln(2d-d^2)$ is a fixed negative number and
  $\ln(n)$ can be made arbitrarily large,
  hence $\lim_{n \to \infty}\|x_n - e_1^*(x_n)e_1\| = 1$.
  It follows that $h_M$ is not uniformly strictly monotone
  at any coordinate.
\end{exmp}

As proved in \cite[III.Examples~1.4]{MR1238713}, if $M$ is an Orlicz
function which fails the $\Delta_2$-condition at zero, while its
complementary function $M^*$ satisfies it, then the Orlicz sequence
space $h_M$ is M-embedded and thus has the LD2P. The same conclusion
holds for the preduals $d(w, 1)_*$ of the Lorentz sequence space
$d(w, 1)$.  Example~\ref{ex:orlicz-SC1symLD2P} shows that there exist
strictly convex Orlicz sequence spaces $h_M$ with a 1-symmetric basis
and the LD2P. As for Lorentz sequences spaces $d(w, 1),$ however,
their preduals $d(w, 1)_*$ are never strictly convex since their
canonical basis is monotone, but not strictly monotone.
We will come back to these preduals once more in the
next section.

\begin{rem}
  In general, $\mathfrak{L_F} = \infty$ does not tell us anything
  about uniform strict monotonicity at a coordinate or the LD2P.

  Any Orlicz space with $M$ non-degenerate
  (hence $M$ strictly increasing)
  satisfies $\mathfrak{L_F} = \infty$.
  The Orlicz space $h_M$ in Example~\ref{ex:orlicz-SC1symLD2P}
  has the LD2P and is not uniform strict monotonicity at any coordinate.
  On the other hand, $\ell_p$, $1 \le p < \infty$,
  clearly does not have the LD2P.
  If $x \in B_{\ell_p}$ with $e_1^*(x) = (1-d)$ we have
  $\|x - e_{1}^*(x)e_1\| = (1 - (1-d)^p)^{1/p}$
  which for small $d$ is roughly $(pd)^{1/p}$
  so $\ell_p$ is uniformly strictly monotone at a coordinate.

  Furthermore, the diameter of the slice $S(e_i^*,\delta)$
  does not depend on the rate at which
  $\mathfrak{L}_n = \|\sum_{i=1}^n e_i\|$ grows.
  Indeed, for $h_M$ in Example~\ref{ex:orlicz-SC1symLD2P} we have
  $\mathfrak{L}_n \approx \ln(n)$,
  for $\ell_p$ we have $\mathfrak{L}_n = n^{1/p}$,
  while for the predual $d(w,1)_*$ where $w_n = 1/n$
  we have $\mathfrak{L}_n \approx n/\ln(n)$.
\end{rem}

\section{Strictly convex renormings with the LD2P}
\label{sec:rotund-norm-c_0gamma}
In the first part of the paper, our aim was to find properties on
strictly convex solid Banach lattices of real functions, that
prevented such spaces from having the LD2P. In particular, we found
that any strictly convex norm on $c_0(\Gamma)$ with a 1-symmetric
basis, fails the LD2P. One of the main results in this section, tells
us that this is close to optimal, even in the non-separable case,
since $c_0(\Gamma),$ for any set $\Gamma,$ will be shown to admit a
strictly convex renorming with a 1-unconditional basis and the
LD2P.
In the separable case $c_0(\mathbb N)$ with the Nakano norm
is such a renorming.
The Nakano renorming of $c_0(\mathbb{N})$ will play a crucial
role in the proofs below and it will for example allow us to
construct strictly convex renormings with the LD2P of any Banach space
containing complemented copies of $c_0(\mathbb N).$

We start by a general construction of a norm. Let
$\varphi: \ell_\infty(\mathbb N) \to [0, \infty]$ be the function
given by $\varphi(x) = \sum_{n = 1}^\infty x_n^{2n}$ where $x = (x_n).$ Let
$(X, \|\cdot\|)$ be any Banach space. Define the function
$\Phi: X^{**} \oplus_\infty \ell_\infty(\mathbb N) \to [0, \infty]$ by
\begin{align}
  \label{eq:17}
  \Phi(z) = \|x\| + \varphi(y),
\end{align}
where $z = (x, y)$ and $\|\cdot\|$ denotes the canonical norm in the
bidual of $X.$ To avoid heavy use of parentheses, we will sometimes
just write $\Phi(x, y)$ instead of $\Phi((x, y)),$ $\|\cdot\|$ instead
of $\|(\cdot, \cdot)\|$ when $\|\cdot\|$ is a norm on
$X^{**} \oplus \ell_\infty(\mathbb N),$ and $Px$ instead of $P(x)$
when $P$ is an operator on $X^{**}.$ For
  \begin{align*}
    Z_\infty &:= (X \oplus_\infty c_0(\mathbb N), \|(\cdot, \cdot)\|_\infty),
  \end{align*}
  we have
  \begin{align*}
    Z_\infty^{**}
    &= (X^{**} \oplus_\infty \ell_\infty(\mathbb N),
      \|(\cdot, \cdot)\|_\infty).
  \end{align*}
We start by collecting some properties of $\Phi.$

\begin{lem}
  \label{lem:phi-properties-1}
  Let $\Phi$ be the function given in (\ref{eq:17}).
  We have:
  \begin{enumerate}
  \item\label{phi-properties-1-pt1}
    The set
    \[
      C = \{(x, y) \in X^{**} \oplus_\infty \ell_\infty(\mathbb N):
      \Phi(x, y) \le 1\},
    \]
    is convex, symmetric, and absorbing with $\frac{1}{2}B_{Z_\infty^{**}}
    \subseteq C \subseteq B_{Z_\infty^{**}}.$

    In particular, if $(x, y) \in C,$ then $\|(x, y)\|_{\infty} \le
    \Phi(x, y).$
  \item\label{phi-properties-1-pt2}
    The Minkowski functional on $C$ is a norm
    $\vertiii{(\cdot, \cdot)}$ given by
    \[
      \vertiii{(x, y)} = \inf\{\lambda > 0: \Phi((x, y)/\lambda) \le
      1\}.
    \]
  \item\label{phi-properties-1-cont}
    $\Phi$ is continuous on $Z_\infty.$
  \item\label{phi-properties-1-submul}
    For each $(x, y) \in Z_\infty^{**},$
    if $\lambda > 1,$ then
    \begin{equation*}
      \Phi((x, y)/\lambda) \le \frac{1}{\lambda}\Phi(x, y).
    \end{equation*}
  \item\label{phi-properties-1-sphere}
    For each $(x, y) \in Z_\infty$,
    if $\vertiii{(x, y)} = 1$, then $\Phi(x, y) = 1$.
  \end{enumerate}
\end{lem}

\begin{proof}
  \ref{phi-properties-1-pt1}.
  Since $\|\cdot\|$ and $\varphi$ are convex and symmetric, $C$
  is convex and symmetric as well.

  If $\|(x, y)\|_\infty \le 1/2,$ then
  \[
    \Phi(x, y)
    = \frac{1}{2} + \sum_{n = 1}^\infty \frac{1}{2^{2n}}
    = \frac{1}{2} + \frac{1}{3}
    < 1,
  \]
  so $\frac{1}{2}B_{Z_\infty^{**}} \subseteq C$ from which we can also
  infer that $C$ is absorbing.

  If $(x, y) \in \lambda C,$ then
  \[
    \frac{1}{\lambda}\|x\|
    \le \Phi((x,y)/\lambda)
    \le 1,
  \]
  so $\|x\| \le \lambda.$ Similarly,
  \[
    \frac{1}{\lambda^{2n}}|y_n|^{2n}
    \le \varphi(y/\lambda)
    \le \Phi((x, y)/\lambda)
    \le 1
  \]
  and $|y_n| \le \lambda$ for all $n \in \mathbb N.$ In particular, if
  $\lambda = 1,$ we get $\|(x, y)\|_\infty \le \Phi(x, y).$

  \ref{phi-properties-1-pt2}.
  The Minkowski functional on $C$ is given by
  \[
    \vertiii{(x, y)}
    = \inf\{\lambda > 0: (x, y) \in \lambda C\},
  \]
  which means that $(x, y)/\lambda \in C.$
  The rest follows from \ref{phi-properties-1-pt1}.

  \ref{phi-properties-1-cont}.
  We have that
  \begin{equation*}
    (x,y) \to x \to \|x\|
  \end{equation*}
  is continuous and that for any $N$
  \begin{equation*}
    (x,y) \to y \to \sum_{n=1}^N y_n^{2n}
  \end{equation*}
  is continuous.
  It follows that $\Phi(\cdot, \cdot)$ is continuous
  on a dense subset of $Z_\infty$ which is enough.

  \ref{phi-properties-1-submul}.
  We have
  \begin{equation*}
    \Phi((x,y)/\lambda)
    =
    \frac{1}{\lambda}\|x\|
    +
    \sum_{n=1}^\infty (y_n/\lambda)^{2n}
    \le
    \frac{1}{\lambda}
    \Phi(x,y).
  \end{equation*}

  \ref{phi-properties-1-sphere}.
  Let $z = (x,y) \in Z_\infty$ with $\vertiii{(x,y)} = 1$.
  By \ref{phi-properties-1-cont} $\Phi$ is continuous on $Z_\infty$
  so that
  \begin{equation*}
    \Phi(z) = \lim_{\lambda \to 1^+} \Phi(z/\lambda)
    \le 1.
  \end{equation*}
  Assume for contradiction that $\Phi(z) < 1$.

  Find $z^* = (x^*, y^*)$ with $\vertiii{(x^*,y^*)} = 1$
  such that
  \begin{equation*}
    z^*(z) = x^*(x) + y^*(y) = 1.
  \end{equation*}
  Note that $x^*(x) \ge 0$ and $y^*(y) \ge 0.$ Indeed, if
  $x^*(x) < 0,$ then since $\Phi(-x, y) = \Phi(x, y) \le 1$ we have
  $\vertiii{(-x, y)} \le 1,$ while
  $x^*(-x) + y^*(y) > x^*(x) + y^*(y) = 1.$ Similarly, we get a
  contradiction if $y^*(y) < 0.$

  Consider first the case where $y^*(y) > 0.$
  We can write
  $y = (y_n)_{n = 1}^\infty$ and $y^* = (y_n^*)_{n = 1}^\infty$.
  Note that $|y_n| < 1$ for all $n \in \mathbb{N}$.
  Choose $k \in \mathbb{N}$
  such that $y_k^* \cdot y_k \ne 0$ and find $t > |y_k|$ with
  \begin{equation*}
    \Phi(z) + t^{2k} - (y_k)^{2k} \le 1.
  \end{equation*}
  Define $w = (x, (v_n)_{n = 1}^\infty)$ where
  $v_n = y_n$ for $n \ne k$ and $v_k = t\cdot \sgn(y_k^*)$.
  Since
  \begin{equation*}
    \Phi(w)
    = \Phi(z) + t^{2k} - (y_k)^{2k}
    \le 1,
  \end{equation*}
  we get $\vertiii{w} \le 1.$
  On the other hand
  \begin{equation*}
    z^*(w) - z^*(z)
    = |y_k^*| \cdot t - y_k^* \cdot y_k
    > 0.
  \end{equation*}
  So $1 \ge z^*(w) > z^*(z) = 1$, which is a contradiction.

  Consider now the case $x^*(x) > 0.$ Choose $t > 1$ such that
  \begin{equation*}
    \Phi(z) + (t - 1)\|x\|
    \le 1.
  \end{equation*}
  Define $w = (tx, y).$ Since
  \begin{equation*}
    \Phi(w)
    = \Phi(z) + (t - 1)\|x\|
    \le 1,
  \end{equation*}
  we get $\vertiii{w} \le 1.$
  On the other hand
  \begin{equation*}
    z^*(w) - z^*(z)
    = tx^*(x) + y^*(y) - (x^*(x) + y^*(y))
    = (t - 1)x^*(x)
    > 0.
  \end{equation*}
  So $1 \ge z^*(w) > z^*(z) = 1$, which is a contradiction.

  We conclude that $\Phi(x,y) = 1$ as desired.
\end{proof}

Assume for the moment that our $X$ is a strictly convex Banach space
and has a normalized shrinking 1-unconditional basis
$(e_\gamma)_{\gamma \in \Gamma}.$
Note that for $e_\gamma \in X$ and $e_n \in c_0(\mathbb N),$ we have
\begin{align*}
  \vertiii{(e_\gamma, 0)}
  = \vertiii{(0, e_n)}
  = 1.
\end{align*}
Let $\mathcal A = \Gamma \cup \mathbb N$ and define by
$f_\alpha = (e_\gamma, 0)$ when $\alpha = \gamma \in \Gamma$ and
$f_\alpha = (0, e_n)$ when $\alpha = n \in \mathbb N.$ Since
$(e_n)_{n \in \mathbb N}$ is a shrinking 1-unconditional basis for
$c_0(\mathbb N),$ we get by a standard argument that
$(f_\alpha)_{\alpha \in \mathcal A}$ is a shrinking 1-unconditional
basis for $Z_\infty.$
Note that if $(x, y), (u, v) \in Z_\infty^{**}$ and $|(x, y)| \le |(u, v)|$,
meaning
$|x_\gamma| \le |u_\gamma|$ for all $\gamma \in \Gamma$
and
$|y_n| \le |v_n|$ for all $n \in \mathbb{N}$,
then $\Phi(x, y) \le \Phi(u, v)$.

In the presence of a 1-unconditional basis
$\Phi$ will have some additional properties.

\begin{lem}\label{lem:phi-properties-2}
  If $X$ is a strictly convex Banach space with a
  normalized shrinking 1-unconditional basis $(e_\gamma)_{\gamma \in \Gamma},$
  then the function $\Phi$ given in (\ref{eq:17})
  satisfies the following properties:
  \begin{enumerate}
  \item\label{phi-properties-2:bidual-elem}
    For each $(x, y) \in Z_\infty^{**},$ we have
    \begin{equation*}
      \Phi(x, y) = \sup_G \Phi(P_G^{**}(x, y)),
    \end{equation*}
    where $P_G: Z_\infty^{**} \to Z_\infty^{**}$ denotes the natural
    projection onto the finite dimensional subspace
    $\linspan(f_\alpha)_{\alpha \in G} \subset Z_\infty^{**}$ and where
    $G$ is a finite subset of $\mathcal{A}.$
  \item\label{phi-properties-2:bidual-sphere}
    For each $(x, y) \in Z_\infty^{**} ,$
    if $\vertiii{(x, y)} = 1,$ then $\Phi(x, y) \le 1.$
  \end{enumerate}
\end{lem}

\begin{proof}
  \ref{phi-properties-2:bidual-elem}.
  Let
  $\mathcal F = \{F \subset \Gamma: |F| < \infty\}$ and
  $\mathcal G = \{G \in \mathcal A: |G| < \infty\}$. Order
  $\mathcal F$ by set inclusion and $\mathcal G$ by the rule
  $F_1 \times \{1, \ldots N_1\} \le F_2 \times \{1, \ldots N_2\}$ if
  $F_1 \subseteq F_2$ and $N_1 \le N_2.$ Since
  $(e_\gamma^*)_{\gamma \in \Gamma}$ is a shrinking 1-unconditional
  basis for $Z_\infty$ and
  $P_G^{**}(x, y) = (P_F^{**}x, P_N^{**}y)$ for
  every $G = F \times \{1, \ldots, N\} \subset \mathcal A,$ we get using
  Proposition~\ref{prop:shrinking-1unc-sup} that
  \begin{align*}
    \Phi(x, y)
    &= \|x\| + \varphi(y)\\
    &= \sup_{F \in \mathcal F} \|P_F^{**}x\| + \sup_{N \in \mathbb N}
      \varphi(P_N^{**}y)\\
    &= \lim_{\mathcal F} \|P_F^{**}x\| + \lim_{N \to \infty}\varphi(P_N^{**}y) \\
    &= \lim_{\mathcal G} \Phi(P_G^{**}(x,y))
      = \sup_{G}\Phi(P_G^{**}(x,y)).
  \end{align*}

  \ref{phi-properties-2:bidual-sphere}.
  Using \ref{phi-properties-2:bidual-elem}
  and Lemma~\ref{lem:phi-properties-1}~\ref{phi-properties-1-cont}
  we have
  \begin{align*}
    \Phi(x,y)
    &=
    \sup_G \Phi(P_G^{**}(x,y))
    =
    \sup_G \inf_{\lambda > 1} \Phi(P_G^{**}((x,y)/\lambda)) \\
    &\le
    \sup_G \inf_{\lambda > 1} \Phi((x,y)/\lambda)
    \le
    1.
  \end{align*}
  Note that, for the second equality,
  we used $P_{G}^{**}(x,y) \in Z_\infty$,
  which holds by equation \eqref{eq:15} in
  Proposition~\ref{prop:shrinking-1unc-sup}.
\end{proof}

\begin{thm}\label{thm:gen-uncountable-propA1}
  Let $X$ be a strictly convex Banach space
  and let
  \begin{equation*}
    Z := (X \oplus_\infty c_0(\mathbb N), \vertiii{\cdot})
  \end{equation*}
  be the renorming of $Z_\infty$ defined above.

  The following statements are true:
  \begin{enumerate}
  \item\label{item:thm-gup-sc}
    The Banach space $Z$ is strictly convex.
  \item\label{item:thm-gup-asq}
    The Banach space $Z$ is almost square.
  \item\label{item:thm-gup-ld2p}
    The Banach space $Z$ has the LD2P.
  \end{enumerate}
  If, in addition, $X$ has a shrinking 1-unconditional basis,
  then the following statements are true:
  \begin{enumerate}[resume]
  \item\label{item:thm-gup-1uncond}
    The Banach space $Z$ has a 1-unconditional basis.
  \item\label{item:thm-gup-bidual}
    The Banach space $(X^{**} \oplus_\infty \ell_\infty(\mathbb N),
    \vertiii{\cdot})$ is the bidual of
    $Z.$
  \end{enumerate}
\end{thm}

\begin{proof}
  \ref{item:thm-gup-sc}.
  Let $a,b \in S_Z$.
  Assume that $\frac{1}{2}\vertiii{a + b} = 1$.
  From Lemma~\ref{lem:phi-properties-1}~\ref{phi-properties-1-sphere}
  we have
  \begin{equation*}
    \Phi(a)
    = \Phi(b)
    = \Phi\left(\frac{1}{2}(a +  b)\right)
    = 1.
  \end{equation*}
  Thus
  \begin{align*}
    &
    \frac{
      \Phi\left(
      \frac{1}{2}(a +  b)
      +
      \frac{1}{2}(a -  b)
      \right)
      +
      \Phi\left(
      \frac{1}{2}(a +  b)
      -
      \frac{1}{2}(a -  b)
      \right)
    }{2}\\
    &=
    \frac{\Phi(a) + \Phi(b)}{2}
    = \Phi\left(\frac{1}{2}(a + b)\right)
  \end{align*}
  It is enough to show that $\frac{1}{2}(a - b) = 0$.

  To this end it is enough to show that for all
  $z, w \in Z$ we have that
  \begin{equation*}
    \frac{\Phi(z + w) + \Phi(z - w)}{2} > \Phi(z)
  \end{equation*}
  whenever $w \neq 0$.

  Let $z = (x,y), w = (u,v) \in Z$
  and assume $w \neq 0$.
  We always have
  \begin{align*}
    &\frac{1}{2}\left[\Phi(z + w) + \Phi(z - w)\right] - \Phi(z)\\
    &= \left[\frac{\|x + u\| + \|x - u\|}{2} -
      \|x\|\right]
    + \sum_{n = 1}^\infty\left[\frac{(y_n + v_n)^{2n} +
        (y_n - v_n)^{2n}}{2} -  y_n^{2n}\right]\\
    &\ge 0
  \end{align*}
  because the functions $s \mapsto \|s\|$ and
  $f_n(t) = t^{2n}, n \in \mathbb{N},$ are convex,
  hence all expressions in the brackets are non-negative.
  If $w \ne 0,$ we either have $u \ne 0,$ and thus
  \begin{align*}
    \frac{\|x + u\| + \|x - u\|}{2} -
    \|x\|
    > 0
  \end{align*}
  since $\|\cdot\|$ is strictly convex on $X,$ or we have that
  there exists $n \in \mathbb N$ such that $v_n \ne 0,$ and thus
  \begin{align*}
    \frac{(y_n + v_n)^{2n} +
      (y_n - v_n)^{2n}}{2} -  y_n^{2n}
    > 0
  \end{align*}
  since $f_n$ is strictly convex.
  In either case we get
  \begin{equation*}
    \frac{\Phi(z + w) + \Phi(z - w)}{2} - \Phi(z)
    > 0,
  \end{equation*}
  as desired.

  \ref{item:thm-gup-asq}.  Put
  $\Delta = \linspan\{e_\gamma\}_{\gamma \in \Gamma} \oplus_\infty
  c_{00}(\mathbb{N}).$ Since $\Delta$ is dense in $Z,$ it suffices to
  prove that for $z_1, \ldots z_n \in \Delta \cap S_Z$ and $\eps > 0,$
  there exists $h \in S_Z$ such that $\|z_i + h\| < 1 + \eps$ for
  every $i = 1, \ldots, n.$ To this end write $z_i = (x_i, y_i)$ for
  $i \in \{1, \ldots, n\}.$ Let $1 - \eps/2< r < 1$ and find
  $N > \max\{\supp y_i: i = 1, \ldots, n\}$ such that
  $\varphi(y_i + re_N) < \varphi(y_i) + \eps/2$ for every
  $i = 1, \ldots, n.$ Then for $g= (0, re_N)$ we get
  \[
    \Phi(z_i + g)
    = \|x_i\| + \varphi(y_i + re_N)
    \le \|x_i\| + \varphi(y_i) + \eps/2
    = 1 + \eps/2,
  \]
  so $\Phi((z_i + g)/(1 + \eps/2)) \le 1$
  by Lemma~\ref{lem:phi-properties-1}~\ref{phi-properties-1-submul}
  and hence $\vertiii{z_i + g} \le 1 + \eps/2.$
  Put $h = g/\vertiii{g}.$ Then
  \begin{align*}
    \vertiii{z_i + h}
    \le \vertiii{z_i + g} + \vertiii{h - g}
    \le 1 + \eps/2 + (1 - r)
    < 1 + \eps,
  \end{align*}
  which is want we wanted.

  \ref{item:thm-gup-ld2p}.
  Almost squareness is known to imply the LD2P
  (see e.g \cite[Theorem~1.3 and Proposition~2.5]{MR3415738}).
  Hence \ref{item:thm-gup-ld2p}
  follows from \ref{item:thm-gup-asq}.

  \ref{item:thm-gup-1uncond}.
  As already noted $\vertiii{f_\alpha} = 1$ for all
  $\alpha \in \Gamma \times \mathbb N.$ To see that
  $(f_\alpha)_{\alpha \in \mathcal A}$ is a 1-unconditional shrinking
  basis for $Z= (X \oplus_\infty c_0(\mathbb N), \vertiii{\cdot}),$ we
  first observe that $\clinspan(f_\alpha) = Z.$ Moreover, if
  $|b_\gamma| \le |a_\gamma|$ for $\gamma \in F \subset \Gamma$ and
  $|b_n| \le |a_n|$ for $n \in E \subset \mathbb N$, with
  $|E|,|F| <\infty$, then for
  \begin{equation*}
    u_b = \sum_{\gamma \in F} b_\gamma f_\eta
    \quad \mbox{and} \quad
    v_b = \sum_{n \in E} b_n f_n
  \end{equation*}
  and
  \begin{equation*}
    u_a = \sum_{\gamma \in F} a_\gamma f_\eta
    \quad \mbox{and} \quad
    v_a = \sum_{n \in E} a_n f_n
  \end{equation*}
  we have
  \begin{equation*}
    \Phi((u_b,v_b)/\lambda) \le \Phi((u_a,v_a)/\lambda)
  \end{equation*}
  for all $\lambda > 0$.
  Hence
  $\vertiii{(u_b,v_b)} \le \vertiii{(u_a,v_b)}$,
  and thus $(f_\alpha)_{\alpha \in \mathcal{A}}$
  is a $1$-unconditional basis for $Z$.

  \ref{item:thm-gup-bidual}.
  Let $Y :=
  (
  X^{**} \oplus_\infty \ell_\infty(\mathbb{N}),
  \vertiii{\cdot}
  ).$
  By Remark~18.2 in \cite[p.~33]{MR4501243} it is enough
  to show that the weak$^*$ closure of $B_Z$ in $Y$ is $B_Y$
  (be aware that there is a typo in that remark).
  Let $z_\beta = (x_\beta,y_\beta) \in B_Z$.
  Assume $z = (x,y) \in Y$ such that $z_\beta \to z$ weak$^*$.
  We need to show that $\vertiii{z} \le 1$.

  Let $G \subset \mathcal{A}$ with $|G| < \infty$.
  We have
  \begin{equation*}
    \|P_Gz_\beta - P_G^{**}z\|_\infty \to_{\beta} 0,
  \end{equation*}
  where $P_G$ denotes the projection onto the
  finite dimensional space
  $\linspan(f_\alpha)_{\alpha \in G} \subset Z$.
  From Lemma~\ref{lem:phi-properties-1}~\ref{phi-properties-1-cont}
  we get
  \begin{equation*}
    \Phi(P_Gz_\beta)
    \to_{\beta}
    \Phi(P_G^{**}z).
  \end{equation*}
  By Lemma~\ref{lem:phi-properties-2}~\ref{phi-properties-2:bidual-sphere}
  (and 1-unconditionality of the basis) we have
  \begin{equation*}
    \Phi(P_G z_\beta) \le 1
    \ \mbox{so that}
    \ \Phi(P_G^{**} z) \le 1.
  \end{equation*}
  Since $G \subset \mathcal{A}$ with $|G| < \infty$
  was arbitrary we get
  \begin{equation*}
    \Phi(z) = \sup_G\Phi(P_G^{**} z) \le 1
  \end{equation*}
  from Lemma~\ref{lem:phi-properties-2}~\ref{phi-properties-2:bidual-elem},
  and thus $\vertiii{z} \le 1$.
\end{proof}

The following result should be compared with the main result of
\cite{2024arXiv240803737C} saying that a Banach space which admits a
smooth norm and contains a complemented copy of $\ell_1(\mathbb N)$
has an equivalent norm which is simultaneously smooth and octahedral.

\begin{cor}\label{cor:renorm_str_conv_asq}
  If a Banach space is strictly convex renormable and contains a complemented
  copy of $c_0(\mathbb N)$, then it has an equivalent norm which
  is both strictly convex and almost square.

  In particular, any separable Banach space $X$ which contains a copy
  of $c_0,$ admits a norm which is simultaneously strictly convex and
  almost square.
\end{cor}

\begin{proof}
  A Banach space $Y$ with a strictly convex norm containing
  a complemented copy of $c_0(\mathbb N)$ is isomorphic to
  \begin{equation*}
    Z_\infty = X \oplus_\infty c_0
  \end{equation*}
  where $X$ is a subspace of $Y$.
  Since the norm of $Y$ is strictly convex
  the same holds for the norm of $X$.
  The result now follows from Theorem~\ref{thm:gen-uncountable-propA1}.

  The particular case follows from the first part and Sobczyk's lemma.
\end{proof}

\begin{cor}\label{cor:renorm_c_0_gamma}
  For any infinite set $\Gamma$,
  the Banach space $c_0(\Gamma)$ has a strictly convex
  renorming which is almost square and
  has a 1-unconditional basis.
\end{cor}

\begin{proof}
  The result follows from Theorem~\ref{thm:gen-uncountable-propA1}
  if we just note that
  $c_0(\Gamma) = c_0(\Gamma) \oplus_\infty c_0$
  and that we can equip $c_0(\Gamma)$ with a strictly convex
  norm (e.g. Day's norm \cite[Chapter~9.4.1]{MR4501243}).
\end{proof}

As mentioned earlier, the
predual $d(w, 1)_*$ of the Lorentz space $d(w, 1)$ is never strictly
convex. However, we do have the following result.

\begin{cor}
  \label{cor:lorentz-predual}
  The predual $d(w, 1)_*$ of the Lorentz space $d(w, 1),$ admits a
  strictly convex renorming with a 1-unconditional basis and the LD2P.
\end{cor}

\begin{proof}
  Since $d(w, 1)_*$ is separable with a 1-symmetric basis, we get from
  \cite[Proposition~4]{MR1348481} that this space has an equivalent
  1-symmetric strictly convex renorming $Y$. Since $d(w, 1)_*$ is
  non-reflexive and M-embedded, $Y$ contains a complemented copy of
  $c_0(\mathbb N)$ by \cite[Corollary~4.7 (d)]{MR1238713}.  Hence, denoting this
  complement by $X,$ which as a subspace of $Y$ is strictly convex,
  Theorem \ref{thm:gen-uncountable-propA1} applies to conclude
  that $d(w, 1)_*$ admits a strictly convex renorming with a
  1-unconditional basis and the LD2P.
\end{proof}

Let us end the paper with some questions.

\begin{quest}
  Does the predual $d(w, 1)_*$ of the Lorentz space $d(w, 1)$ admit a
  strictly convex renorming with the LD2P for which the canonical
  basis of $d(w, 1)_*$ is 1-symmetric?
\end{quest}

\begin{quest}
  Does there exist a strictly convex renorming of
  $\ell_\infty(\mathbb N)$ with the LD2P?
\end{quest}

In \cite[Proposition~A.1]{MR3499106} it was shown that
$c_0(\mathbb{N})$ with the Nakano norm is not only
strictly convex with the LD2P,
it is even M-embedded and its dual is smooth.
For $\Gamma$ uncountable $\ell_1(\Gamma)$ does not
have an equivalent smooth norm so such a renorming
is not possible for $c_0(\Gamma)$.
However, Corollary~\ref{cor:renorm_c_0_gamma} gives
a strictly convex renorming of $c_0(\Gamma)$
which is almost square and has a 1-unconditional basis
and we do not know if this can be improved to be M-embedded.

\begin{quest}
  Let $\Gamma$ be uncountable.
  Does there exist an M-embedded strictly convex renorming of
  $c_0(\Gamma)$ with a 1-unconditional basis?
\end{quest}

\section*{Acknowledgments}
\label{sec:ack}
Some of this research was carried out when the first author visited
the Faculty of Electrical Engineering at the Czech Technical
University in Prague in the spring of 2024. He would like to thank his
host for the warm hospitality extended during his stay.

The second author was supported by GA23-04776S and
SGS22/053/OHK3/1T/13 of CTU in Prague.

\bibliographystyle{amsalpha}
\newcommand{\etalchar}[1]{$^{#1}$}
\providecommand{\bysame}{\leavevmode\hbox to3em{\hrulefill}\thinspace}
\providecommand{\MR}{\relax\ifhmode\unskip\space\fi MR }
\providecommand{\MRhref}[2]{%
  \href{http://www.ams.org/mathscinet-getitem?mr=#1}{#2}
}
\providecommand{\href}[2]{#2}

\end{document}